\documentclass[11pt]{amsart}      
\usepackage{amsmath,amssymb,amsthm,color,enumerate,enumitem,tikz-cd}
\usepackage{hyperref} 
\usepackage[normalem]{ulem} 
\usepackage[top=1.5in,bottom=1.5in,a4paper]{geometry}
\usepackage{orcidlink}

\usepackage{xargs,soul}
\usepackage{xcolor}
\usepackage[textsize=tiny]{todonotes}
\newcommandx{\unsure}[2]{\todo{#2}\texthl{#1}}

\newtheorem{thm}{Theorem}
\newtheorem{thmx}{Theorem}

\newtheorem{conjecture}[thmx]{Conjecture}
\newtheorem{prop}[thm]{Proposition}
\newtheorem{cor}[thm]{Corollary}
\newtheorem{lemma}[thm]{Lemma}

\newtheorem{fact}[thm]{Fact}

\newtheorem{examples}[thm]{Examples}
\theoremstyle{definition}
\newtheorem{defin}[thm]{Definition}

\theoremstyle{remark}
\newtheorem{remark}[thm]{Remark}

\newcommand{\Rea}{\mathbb R}
\newcommand{\Nat}{\mathbb N}
\newcommand{\Rat}{\mathbb Q}

\newcommand{\A}{\mathcal A}
\newcommand{\F}{\mathcal F}

\newcommand{\B}{\mathcal B}

\def \rng {\operatorname{Rng}}
\def \ker {\operatorname{Ker}}
\newcommand{\Span}{\operatorname{span}}

\def \sgn{\operatorname{sgn}}
 
\def \suppt {\operatorname{suppt}}

\def \dens {\operatorname{dens}}
\def \diam {\operatorname{diam}}
\def \dist {\operatorname{dist}}
\def \conv {\operatorname{conv}}
\def \Lip {\operatorname{Lip}}

\begin{document}
	
	\title{On the weak$^*$ separability of the space of Lipschitz functions}

	\author[L. Candido]{Leandro Candido \orcidlink{0000-0002-6429-3899}}
	\author[M. C\' uth]{Marek C\'uth \orcidlink{0000-0001-6688-8004}}
	\author[B. Vejnar]{Benjamin Vejnar \orcidlink{0000-0002-2833-5385}}
	\email{leandro.candido@unifesp.br}
	\email{cuth@karlin.mff.cuni.cz}
	\email{vejnar@karlin.mff.cuni.cz}
	
	\address[L.~Candido]{Universidade Federal de S\~{a}o Paulo - UNIFESP. Instituto de Ci\^{e}ncia e Tecnologia. Departamento de Matematica. S\~{a}o Jos\'e dos Campos - SP, Brasil}
	
	\address[M.~C\' uth, B.~Vejnar]{Charles University, Faculty of Mathematics and Physics, Department of Mathematical Analysis, Sokolovsk\'a 83, 186 75 Prague 8, Czech Republic}

	
	\thanks{L. Candido was supported by Funda\c c\~ao de Amparo \`a Pesquisa do Estado de S\~ao Paulo - FAPESP no. 2023/12916-1. M. C\'uth was supported by the GA\v{C}R project 23-04776S}
	
	\begin{abstract}We conjecture that whenever $M$ is a metric space of density at most continuum, then the space of Lipschitz functions is $w^*$-separable. We prove the conjecture for several classes of metric spaces including all the Banach spaces with a projectional skeleton, Banach spaces with a $w^*$-separable dual unit ball and locally separable complete metric spaces.
	\end{abstract}
	
	\maketitle
	
	Given a metric space $M$, the Banach space of Lipschitz functions $\Lip_0(M)$ has a natural predual, discovered and re-discovered by several authors and known under several names. The one we use in this paper is the Lipschitz-free space over $M$ denoted by $\F(M)$, others are e.g. the Arens-Eells space or the transportation cost space (we refer to \cite[Subsection 1.6]{OO19} for a more detailed discussion concerning the terminology). The study of Banach-space theoretical properties of Lipschitz-free spaces is an ongoing and very active field of research. In this paper we deal mostly with the case when the metric space $M$ is not separable, which is the topic deeply investigated e.g. by N. Kalton in \cite{Kal11} or more recently by P. H\'ajek and A. Quilis in \cite{HQ22}. Probably the most important well-known and very general property of those spaces is the observation that $\ell_1(\dens M)$ is isomorphic to a complemented subspace of $\F(M)$ proved originally in \cite{HajekNovotny}. There are several consequences of this fact, see e.g. \cite{AGP24, K24} where some of those are collected. An example of an interesting open problem is whether there exists a metric (or even Banach) space $M$ such that $\F(M)$ is not Plichko, see e.g. \cite{GMQ23, HQ23} for some recent closely connected results.
	
	The purpose of our paper is to introduce the following conjecture (when discussing $w^*$ topology on $\Lip_0(M)$, in this paper we always refer to the topology induced by its canonical predual $\F(M)$, see Remark~\ref{rem:wStar} for a discussion concerning other possible choices).
	
	\begin{conjecture}\label{conj:main}If $M$ is a metric space and $\dens M\leq 2^\omega$, then $(\Lip_0(M),w^*)$ is separable.
	\end{conjecture}
	
	In case the answer is positive, we believe it would be a nice result significantly improving our knowledge of spaces of Lipschitz functions and their preduals. In the case the answer is negative, it would provide us with an example of a Lipschitz-free space which is not Plichko (and even without Markushevich basis) as we are able to prove that the answer to Conjecture~\ref{conj:main} is positive whenever $\F(M)$ is Plichko, see Corollary~\ref{cor:MbasisCase}.
	
	In this paper we prove that the answer to Conjecture~\ref{conj:main} is positive in the following cases:
	\begin{itemize}
		\item $M$ is a Banach space with a Markushevich basis (see Corollary~\ref{cor:MbasisCase}),
		\item $M$ is a Banach space with a $w^*$-separable dual unit ball (see Theorem~\ref{thm:ball}),
		\item $M$ is a complete and locally separable metric space (see Theorem~\ref{thm:locSep}).
	\end{itemize}
	The most demanding case is the last one and its proof involves some new general techniques which could potentially be further developed. Concerning its importance, note that any metric space is a subset of the completion of a discrete set of the same density, see Fact~\ref{fact:discreteSubset}. Thus, if we could drop the assumption of completeness, this would yield the general answer.
	
	Recall that $w^*$ topology on bounded sets in $\Lip_0(M)$ coincides with the topology of pointwise convergence $\tau_p$. The next result, indicating that Conjecture~\ref{conj:main} could hold, is that $(\Lip_0(M),\tau_p)$ is separable whenever $M$ is a metric space of density at most continuum, see Theorem~\ref{thm:tauPMain} (this particular result was independently proved also in the very recent preprint \cite{BKS24}, see Remark~\ref{rem:kakolTauP} for a more detailed comment).
	
	We note that our results give us quite many examples of metric spaces $M$ (e.g. $M = C([0,\omega_1])$) such that for the Banach spaces $X = \F(M)$ we are able to prove that $X^*$ is $w^*$-separable while $B_{X^*}$ is not $w^*$-separable, see Remark~\ref{rem:sepNotSep}. This seems to be interesting, because in past it was open whether such an example of a Banach space even exists, see \cite{DS79}, and even recently there occur papers dealing with constructions of such spaces, see e.g. \cite{AP14}.
	
	Interested in the $w^*$-separability of Banach spaces, we obtain also two more results. Namely,
	\begin{itemize}
		\item there is a linear isometric embedding of $\left(\bigoplus_{2^{\omega}}\ell_\infty\right)_{\ell_1}$ into $\ell_\infty$ (see Proposition~\ref{prop:ell1Injection});
		\item Lipschitz-free spaces over $\ell_\infty$ and over its $c_0(2^\omega)$-sum $(\bigoplus_{2^{\omega}}\ell_\infty)_{c_0}$ are linearly isomorphic (see Theorem~\ref{thm:c0Andellinfty}).
	\end{itemize}
	Their connection to Conjecture~\ref{conj:main} is e.g. that they imply that $(\Lip_0(M),w^*)$ is separable whenever there exists a bounded linear injective operator $T:\F(M)\to X$, where $X$ is one of the spaces $\left(\bigoplus_{2^{\omega}}\ell_\infty\right)_{\ell_1}$ or $\left(\bigoplus_{2^{\omega}}\ell_\infty\right)_{c_0}$, see Corollary~\ref{cor:reduction}. Also, the first result implies that the class of Banach spaces with a $w^*$-separable duals is preserved by $\ell_1(2^\omega)$-sums, see Corollary~\ref{cor:preservation}.
	
	\begin{remark}\label{rem:wStar}
		As mentioned above, when discussing $w^*$ topology on $\Lip_0(M)$, in this paper we always refer to the topology induced by its canonical predual $\F(M)$. However, it might be the case that this is actually the unique choice as it is open whether any linear isometry between $\Lip_0(M)$ and some $Y^*$, where $Y$ is a Banach space, is actually a $w^*$-$w^*$ homeomorphism, see \cite{Wea18} where it is proved this is indeed the case whenever $M$ is a bounded metric space or whenever $M$ is a Banach space.
	\end{remark}
	
	The structure of this paper is the following: in Section~\ref{sec:kalton} we start with ideas based on an argument by N. Kalton which initiated our interest in Conjecture~\ref{conj:main}. In Section~\ref{sec:pointwise} we concentrate our results connected with the topology of pointwise convergence. In Section~\ref{sec:locSep} we give the proof of probably the most involved result of the whole paper, namely that $(\Lip_0(M),w^*)$ is separable whenever $M$ is a complete and locally separable metric space of density at most continuum (this is motivated by \cite[Section 2]{AGP24}, where the proof for uniformly discrete metric spaces is given). We also exhibit an example of a uniformly discrete metric space for which on the contrary $B_{\Lip_0(M)}$ is not $w^*$ separable. Finally, in Section~\ref{sec:banach} we collect the proof that $\left(\bigoplus_{2^{\omega}}\ell_\infty\right)_{\ell_1}$ embeds linearly isometrically into $\ell_\infty$ and results implying in particular that Lipschitz-free spaces over $\ell_\infty$ and over its $c_0(2^\omega)$-sum are linearly isomorphic.
	
	Our notation is quite standard, for notions concerning Banach spaces we refer to \cite{AlbiacKalton} and for notions concerning general topology to \cite{Eng}. For construction and all the basic facts concerning Lipschitz-free spaces needed in this paper we refer to \cite[Section 2]{CDW}, for more details one may have a look at the authoritative paper \cite{GK03}, monograph \cite{WeaverBook} or the survey \cite{GodefroySurvey}. The only notation which might need to be mentioned is that given a point $x$ in a metric space and $r>0$, we denote by $B(x,r)$ the open ball.
	
	\section{Kalton's argument and its consequences}\label{sec:kalton}
	
	The starting point of our interest in Conjecture~\ref{conj:main} is the following observation which basically follows from the proof of \cite[Theorem 6.4]{Kal11}.
	
	\begin{thm}\label{thm:kalton11}Let $X$ be a Banach space and $Y\subset X$ subspace with $w^*\text{-}\dens Y^*>\omega$. Let there exist an operator $T:X\to c_0(2^\omega)$ such that $T|_Y$ is injective. Then there exists a Banach space $Z$ and a short exact sequence
		\[0\to c_0\to Z\to X\to 0\]
		which Lipschitz splits, but not linearly splits.
	\end{thm}
	
	Before explaining consequences of Theorem~\ref{thm:kalton11} and its proof, let us recall notions used in its statement and certain preliminaries.
	
	Let $X, Y, Z$ be Banach spaces. Then a \emph{short exact sequence} is given by bounded linear operators $T:X\to Y$ and $Q:Y\to Z$ such that $T$ is one-to-one, $Q$ is surjective and $\ker Q = \rng T$. In this case we write
	\[
	0\to X\stackrel{T}{\to} Y\stackrel{Q}{\to} Z\to 0.
	\]
	We say that this short exact sequence \emph{linearly splits}, resp. \emph{Lipschitz splits} if there exists a bounded linear operator, resp. Lipschitz mapping $f:Z\to Y$ such that $Qf=Id$. It is well-known that a short exact sequence linearly splits if and only if $T(X)$ is complemented in $Y$. Moreover, if it linearly (resp. Lipschitz) splits then $Y$ is linearly (resp. Lipschitz) isomorphic to the Banach space $X\oplus_1 Z$. We refer the interested reader to \cite[Chapter 2.1]{CastilloBook} and \cite[Section 2]{GK03} for more information concerning short exact sequences and their splittings. Following \cite{GK03} we say that a Banach space $X$ has the \emph{lifting property} if and only if every short exact sequence $0\to Z\to Y\to X\to 0$ which Lipschitz splits also linearly splits\footnote{This is not the original definition, but it is equivalent to it, see \cite[Proposition 2.8]{GK03}.} and recall that any separable Banach space as well as any Lipschitz-free space has the lifting property (see \cite[Lemma 2.10 and Theorem 3.1]{GK03}).
	
	A useful tool in this context are pull-back constructions. Recall that given Banach spaces $X,Y,Z$ and bounded linear operators $T:X\to Z$, $S:Y\to Z$, the \emph{pull-back of the tuple $(T,S)$} is the Banach space $PB:=\{(x,y)\in X\oplus_\infty Y\colon Tx=Sy\}\subset X\oplus_\infty Y$. The following well-known observation shows how to apply the pull-back construction in the context of short exact sequences (the proof is standard exercise and can be found together with some more information e.g. in \cite[Chapter 2]{CastilloBook}, for the proof of the Lipschitz splitting in the ``Moreover'' part, see e.g. \cite[Lemma 6.1]{Kal11}).
	
	\begin{lemma}\label{lem:pullBackLemma}Let $0\to X\stackrel{d}{\to} Y\stackrel{Q}{\to} Z\to 0$ be a short exact sequence and let $T:W\to Z$ be an operator. Let $PB$ be the pull-back of the tuple $(Q,T)$. Then there exists a short exact sequence $0\to X\stackrel{s}{\to} PB\stackrel{\pi_W}{\to} W\to 0$ such that the following diagram commutes.
		\[
		\begin{tikzcd}
			0 \arrow{r}
			& X \arrow{d}{Id} \arrow{r}{s} 
			& PB \arrow{d}{\pi_Y} \arrow{r}{\pi_W}
			& W \arrow{d}{T} \arrow{r}
			& 0\\
			0 \arrow{r}
			& X \arrow{r}{d} 
			& Y \arrow{r}{Q}
			& Z \arrow{r}
			& 0
		\end{tikzcd}
		\]
		Moreover, if the short exact sequence $0\to X\stackrel{d}{\to} Y\stackrel{Q}{\to} Z\to 0$ Lipschitz splits (resp. linearly splits), then the short exact sequence $0\to X\stackrel{s}{\to} PB\stackrel{\pi_W}{\to} W\to 0$ Lipschitz splits (resp. linearly splits).
	\end{lemma}
	
	Finally, let us recall that $JL_\infty$ is a Banach space with weak$^*$ separable dual for which there exists a short exact sequence $0\to c_0\to JL_\infty\to c_0(2^\omega)\to 0$ which Lipschitz splits, we refer the interested reader e.g. to \cite[Section 6]{Kal11} and references therein for more details.
	
	Let us now recall the argument by N. Kalton together with its consequences.
	
	\begin{proof}[Proof of Theorem~\ref{thm:kalton11}]
		Let us consider pull back as in Lemma~\ref{lem:pullBackLemma}:
		\[
		\begin{tikzcd}
			0 \arrow{r}
			& c_0 \arrow{d}{Id} \arrow{r} 
			& PB \arrow{d}{\pi_{JL_\infty}} \arrow{r}{\pi_X}
			& X \arrow{d}{T} \arrow{r}
			& 0\\
			0 \arrow{r}
			& c_0 \arrow{r} 
			& JL_\infty \arrow{r}{Q}
			& c_0(2^\omega) \arrow{r}
			& 0
		\end{tikzcd}
		\]
		Since the short exact sequence $0\to c_0\to JL_\infty\to c_0(2^\omega)\to 0$ Lipschitz splits, by Lemma~\ref{lem:pullBackLemma} the short exact sequence $0\to c_0\to PB\to X\to 0$ Lipschitz splits as well. In order to get a contradiction assume there is a linear operator $F:X\to PB$ satisfying $\pi_X\circ F = Id$. Then for $S:=\pi_{JL_\infty}\circ F$ we have $Q\circ S = Q\circ \pi_{JL_\infty}\circ F = T\circ \pi_X\circ F = T$. Since $T|_Y$ is injective, $S|_Y:Y\to JL_\infty$ is injective operator and since $JL_\infty$ has $w^*$-separable dual, we have that $Y^*$ is $w^*$-separable, contradiction.
	\end{proof}
	
	This argument, even though quite short, has (at least) two interesting consequences.
	
	The first one is not that much related to Conjecture~\ref{conj:main}, but we believe it is worth to be mentioned. Recall that a Banach space $X$ has SCP (separable complementation property) if every separable subspace of $X$ is contained in a complemented separable subspace of $X$, in particular using Sobczyk's theorem we obtain that every isomorphic copy of $c_0$ in $X$ is complemented. Examples of Banach spaces with SCP are spaces admitting a projectional skeleton, in particular all the Plichko spaces and all the WLD (weakly Lindelöf determined) spaces. We recall that the class of WLD spaces is preserved by closed subspaces, any WLD space admits Markushevich basis (and therefore an injective operator $T:X\to c_0(\dens X)$) and also that whenever $X$ is WLD, then $\dens X = \dens (X^*,w^*)$. We refer the interested reader e.g. to \cite[Chapter 17]{KubisBook} and \cite[Section 3.4 and Chapter 5]{HMVZ} for more details concerning WLD spaces, Plichko spaces, spaces with a projectional skeleton and SCP.
	
	\begin{cor}Let $X$ be a nonseparable WLD Banach space containing isomorphic copy of $c_0$. Then there exists a Banach space $Z$ which does not have SCP and it is Lipschitz-isomorphic to $X$.
		
		In particular, $\F(Z)$ is Plichko but $Z$ does not have even SCP.
	\end{cor}
	\begin{proof}Pick a closed nonseparable subspace $Y\subset X$ with $\dens Y\leq 2^\omega$. Since $Y$ is WLD, there exists an injective linear operator $T:Y\to c_0(2^\omega)$ and so by Theorem~\ref{thm:kalton11} we obtain there is a Banach space $Z$ which is Lipschitz isomorphic to $c_0\oplus X$ but contains a non-complemented subspace isomorphic to $c_0$. Thus, $Z$ does not have (SCP) and since $c_0$ is isomorphic to a complemented subspace of $X$, the space $c_0\oplus X$ is linearly isomorphic to $X$ and therefore $X$ is Lipschitz isomorphic to $Z$.
		
		The ``In particular'' part follows since whenever $Z$ is Lipschitz isomorphic to a Plichko Banach space, then $\F(Z)$ is Plichko (see \cite[Corollary 2.9]{HQ22}).
	\end{proof}
	
	The second consequence brings us to Conjecture~\ref{conj:main}.
	
	\begin{cor}\label{cor:MbasisCase}
		Let $M$ be metric space with $\dens M \leq 2^\omega$ such that there is an injective bounded linear operator $T:\F(M)\to c_0(2^\omega)$. Then $(\Lip_0(M),w^*)$ is separable.
		
		In particular, if $M$ is a metric space with $\dens M \leq 2^\omega$ such that $\F(M)$ has Markushevich basis (which is true e.g. whenever $M$ is a Banach space with a projectional skeleton), then $(\Lip_0(M),w^*)$ is separable.
	\end{cor}
	\begin{proof}Since any Lipschitz-free space has the lifting property and $w^*\text{-}\dens \F(M)^*\leq \dens \F(M) = \dens M$, Theorem~\ref{thm:kalton11} implies that $\omega \geq w^*\text{-}\dens \F(M)^* = \dens (\Lip_0(M),w^*)$.
		
		For the ``In particular'' part we note that a Banach space $X$ with a Markushevich basis admits injective linear operator into $c_0(\dens X)$ (see \cite[Theorem 5.3]{HMVZ}), spaces with a projectional skeleton admit Markushevich basis (see \cite[Theorem 1.2]{Kal20}) and that Lipschitz-free space over a Banach space with a projectional skeleton admits a projectional skeleton as well (see \cite[Proposition 2.8]{HQ22}).
	\end{proof}
	
	The argument of N. Kalton can be actually slightly modified to obtain the following, which enables us to remove the use of the space $JL_\infty$ from the proof of Corollary~\ref{cor:MbasisCase}, see Remark~\ref{rem:corFromCor} below.
	
	\begin{thm}\label{thm:kaltonadaptation}Let $X$ and $Y$ be Banach spaces and assume that there is a bounded linear injective mapping $T:Y\to X$.  If $Y$ has the lifting property (which holds e.g. if $Y$ is a Lipschitz-free space), then there is a bounded linear injective mapping $S:Y\to \F(X)$.
	\end{thm}
	\begin{proof}
		Recall that there exists a short exact sequence $0\to W\to \F(X)\stackrel{\beta}{\to} X\to 0$ which Lipschitz splits (see \cite[p. 125]{GK03}). Consider pull back of this sequence given as in Lemma~\ref{lem:pullBackLemma}:
		\[
		\begin{tikzcd}
			0 \arrow{r}
			& W \arrow{d}{Id} \arrow{r} 
			& PB \arrow{d}{\pi_{\F(X)}} \arrow{r}{\pi_{Y}}
			& Y \arrow{d}{T} \arrow{r}
			& 0\\
			0 \arrow{r}
			& W \arrow{r} 
			& \F(X) \arrow{r}{\beta}
			& X \arrow{r}
			& 0
		\end{tikzcd}
		\]
		Since the short exact sequence $0\to W\to \F(X)\to X\to 0$ Lipschitz splits, by Lemma~\ref{lem:pullBackLemma} the short exact sequence $0\to W\to PB\to Y\to 0$ Lipschitz splits as well. Since $Y$ has the lifting property, we have that $0\to W\to PB\to Y\to 0$ linearly splits. Let $F:Y\to PB$ be a bounded linear mapping such that $\pi_{Y}\circ F = Id$. The desired operator is given by $S:=\pi_{\F(X)}\circ F$. We note that $\beta\circ S = \beta\circ \pi_{\F(X)}\circ F = T\circ \pi_{Y}\circ F = T$ which shows that $S$ is is injective.
	\end{proof}
	
	\begin{cor}\label{cor:intoReduction}
		Let $X$ be a Banach space and $M$ be a metric space such that there is a bounded linear injective mapping 
		$T:\F(M)\to X$. If $(\Lip_0(X),w^*)$ is separable, then  $(\Lip_0(M),w^*)$ is separable.
	\end{cor}
	\begin{proof}
		By Theorem \ref{thm:kaltonadaptation} there exists an injective operator $S:\F(M)\to \F(X)$. Thus, since $S^*:\Lip_0(X)\to \Lip_0(M)$ is $w^*$-$w^*$ continuous and its range is $w^*$-dense, if $(\Lip_0(X),w^*)$ is separable, then  $(\Lip_0(M),w^*)$ is separable.
	\end{proof}
	
	\begin{remark}\label{rem:corFromCor}One may deduce Corollary~\ref{cor:MbasisCase} from Corollary~\ref{cor:intoReduction}. The only fact which we need to realize is that $(\Lip_0(c_0(2^\omega),w^*)$ is separable which follows from well-known results. Indeed, $c_0(2^\omega)$ is Lipschitz isomorphic to a subset of $\ell_\infty$ by the result of \cite{AL78} (see also \cite[Theorem 5.4]{Kal11} for a more general result), so $\F(c_0(2^\omega))$ is linearly isomorphic to a subspace of $\F(\ell_\infty)$ which in turn embeds isometrically into $\ell_\infty$ (see \cite[Proposition 5.1]{Kal11}), so there is an injective linear operator $T:\F(c_0(2^\omega))\to \ell_\infty$ and since $\ell_\infty$ has $w^*$-separable dual, the same holds for $\F(c_0(2^\omega))$.
	\end{remark}
	
	\section{Topology of pointwise convergence}\label{sec:pointwise}
	
	We have observed that $(\Lip_0(X),w^*)$ is separable whenever $X$ is a Banach space with a projectional skeleton and $\dens X\leq 2^\omega$ (see Corollary~\ref{cor:MbasisCase}). In this section we present two more results which indicate the answer to Conjecture~\ref{conj:main} could be positive. Recall that $w^*$-topology on the space of Lipschitz functions coincides with the topology of pointwise convergence. The first main result of this section is the following, which e.g. shows that condition (iv) in \cite[Proposition 2.10]{AGP24} is satisfied in every metric space of density continuum.
	
	\begin{thm}\label{thm:tauPMain}Let $M$ be a metric space with $\dens M\leq 2^\omega$. Then $(\Lip_0(M),\tau_p)$ is separable.
	\end{thm}
	
	\begin{remark}\label{rem:cpm}Note that trivially $\Lip_0(M)$ is a dense subset of $C_p(M):=(C(M),\tau_p)$ for any metric space $M$ (because given values of a continuous function on a finite set, the resulting function defined on this finite set is Lipschitz and we may extend it to the whole space $M$). Thus, from Theorem~\ref{thm:tauPMain} we obtain that $C_p(M)$ is separable whenever $M$ is a metric space with $\dens M\leq 2^\omega$. This result can be deduced also from certain results concerning the theory of $C_p(M)$-spaces\footnote{Indeed, by \cite[Problem 174]{TkachukBook} we have that given a topological space $X$, $C_p(X)$ is separable iff there is a continuous bijection of $X$ into a second countable space and by \cite[Theorem 1.2]{OP23} any metric space of density continuum can be mapped by a continuous bijection onto the Hilbert cube $[0,1]^\omega$.}, but our arguments gives another insight into this.
	\end{remark}
	
	\begin{remark}\label{rem:kakolTauP}
		Theorem~\ref{thm:tauPMain} follows also from the very recent \cite[Theorem 3.8]{BKS24} (where the case of higher cardinalities is considered as well). Our proof is different - its advantage is that it is much more self-contained, so we believe it makes sense to present it here as well. For example, our proof easily implies that $C_p(M)$ is separable, while the proof from \cite{BKS24} is using the fact that $C_p(M)$ is separable (or rather the equivalent condition mentioned in the footnote in Remark~\ref{rem:cpm}) as a tool to prove that $\Lip_0(M)$ is $\tau_p$-separable.
	\end{remark}

	Our second main result of this section is the following.
	
	\begin{thm}\label{thm:ball}Let $X$ be a (non-separable) Banach space. Then
		\[
		\dens \big(B_{\Lip_0(X)},w^*\big)\leq \dens \big(B_{X^*},w^*\big).
		\]
		In particular, $\dens \big(B_{\Lip_0(X)},w^*\big)\leq \dens X$.
	\end{thm}
	
	Applying this result to the space $X = \ell_\infty$ (which does not admit a projectional skeleton), we obtain that not only $\Lip_0(\ell_\infty)$ is $w^*$-separable, but also its unit ball is - this particular consequence follows also from Kalton's \cite[Proposition 5.1]{Kal11}, but our argument is different (and more general) so we again believe it is worth to present the proof here.
	
	Apart from the above mentioned main results as a byproduct of our considerations we obtain also the following.
	
	\begin{prop}\label{prop:dualUnitBall}
		Let $M$ be a metric space. The following conditions are equivalent.
		\begin{enumerate}[label=(\alph*)] 
			\item\label{it:isom} $M$ isometrically embeds into $\ell_\infty(\kappa)$,
			\item\label{it:Fisom} $\F(M)$ linearly and isometrically embeds into $\ell_\infty(\kappa)$,
			\item\label{it:ballSep} $\dens (B_{\Lip_0(M)},w^*)\leq \kappa$.
		\end{enumerate}
	\end{prop}
	
	\begin{remark}
		Note that for $\kappa=\omega$ the equivalence of \ref{it:isom}-\ref{it:ballSep} follows, as observed in \cite[Proposition 2.9]{AGP24}, from \cite[Proposition 5.1]{Kal11} whose proof in turn is based on finding a countable norming set in $B_{\Lip_0(\ell_\infty)}$. Our argument is more general in the sense that we more-or-less explicitly find the norming set with similar properties for any dual space and not only for $\ell_\infty$. See the proof of Proposition~\ref{prop:LipDual} below.
	\end{remark}
	
	The remainder of this section is mostly devoted to proofs of Theorem~\ref{thm:tauPMain}, Theorem~\ref{thm:ball} and Proposition~\ref{prop:dualUnitBall}. Let us start with an essentially known fact proved e.g. in \cite{DS79} for $\kappa=\omega$.
	\begin{fact}\label{fact:denseInDualUnitBall}Let $X$ be a (non-separable) Banach space and $\kappa$ be a cardinal. Then the following conditions are equivalent.
		\begin{itemize}
			\item $w^*\text{-}\dens B_{X^*}\leq \kappa$
			\item There is a linear isometry $T:X\to \ell_\infty(\kappa)$.
		\end{itemize}
	\end{fact}
	\begin{proof}
		If $D\subset (B_{X^*},w^*)$ is dense, then the mapping $T:X\to \ell_\infty(D)$ given by $Tx:=(d(x))_{d\in D}$ is linear isometry.
		
		Conversely, given a linear isometry $T:X\to \ell_\infty(\kappa)$ we consider functionals $f_i\in B_{X^*}$, $i<\kappa$ given by $f_i(x):=T(x)(i)$. Put $D:=\overline{\conv}^{w^*} \{f_i\colon i<\kappa\}$. Since rational convex combinations are linearly dense in $D$, it suffices to prove that $D = B_{X^*}$. Suppose this is not the case, that is, there is $x^*\in B_{X^*}\setminus D$. Then by the Hahn-Banach theorem there is $x\in X$ such that $x^*(x) > \sup\{d(x)\colon d\in D\}$, but since $T$ is isometry we have $\sup\{d(x)\colon d\in D\}\geq \|x\|$ so we obtain $x^*(x) > \|x\|$, which is not possible.
	\end{proof}
	
	Our basic method of recognizing whether a metric embeds into $\ell_\infty(\kappa)$ is based on finding the minimal cardinality of the set of Lipschitz functions separating the points uniformly, this notion was used when studying the geometry of the dual unit ball of Lipschitz-free spaces quite many times, see e.g. \cite[Proposition 3.4]{Kal04}.
	
	\begin{defin}Let $M$ be a metric space. We say $S\subset \Lip_0(M)$ \emph{separates points uniformly (with constant $C\geq 1$)} if for every $x,y\in M$ and $\varepsilon>0$ there is $f\in S$ satisfying $\|f\|\leq C+\varepsilon$ and $|f(x)-f(y)| = d(x,y)$.
	\end{defin}
	
	The following proposition offers another perspective on the previous definition. Recall that a $C$-Lipschitz embedding is an injective Lipschitz mapping $f$ such that $\Lip(f) \cdot \Lip(f^{-1}|_{\mathrm{rng}\, f}) \leq C$.
	
	\begin{prop}\label{prop:EquivalenceEmbeddingSeparatesUniformly}Let $M$ be a metric space and $\kappa$ a cardinal. The following conditions are equivalent.
		\begin{itemize}
			\item There is a $C$-Lipschitz embedding of $M$ into $\ell_\infty(\kappa)$,
			\item There exists a subspace $S \subset \Lip_0(M)$ with density $\kappa$ that separates points uniformly with a constant $C$.
		\end{itemize}
	\end{prop}
	\begin{proof}
		If $S \subset \Lip_0(M)$ is a subspace that separates points uniformly with a constant $C \geq 1$, let $\{f_i : i < \kappa\}$ be a dense subset of $B_S$. Define $T: M \to \ell_\infty(\kappa)$ by the formula $T(x) = (f_i(x))_{i < \kappa}$. It is readily seen that $T$ is a $1$-Lipschitz function with $T(0) = 0$. Given $x, y \in M$ and $\epsilon > 0$, there exists $g \in B_S$ such that $|g(x) - g(y)| = \frac{1}{C + \epsilon} d(x, y)$. Fix $\delta > 0$ and let $i < \kappa$ be such that $\|g - f_i\| \leq \tfrac{\delta}{2}$.
		Hence 
		\begin{align*}
			\|T(x)-T(y)\|\geq |f_i(x)-f_i(y)|\geq |g(x)-g(y)|-\delta d(x,y) = \left(\frac{1}{C+\epsilon}-\delta\right)d(x, y).
		\end{align*}
		Since $\epsilon$ and $\delta$ are arbitrary, we deduce that $T$ is a $C$-Lipschitz embedding.
		
		On the other hand, if there exists an injective $1$-Lipschitz mapping $f: M \to \ell_\infty(\kappa)$ with $f(0) = 0$ and $\Lip(f^{-1}|_{f(M)})$ being $C$-Lipschitz, then it is easy to check that the subspace $S := \Span(\{e_i \circ f : i < \kappa\}) \subset \Lip_0(M)$, where $e_i$ are the coordinate functionals, separates points uniformly with constant $C$.
	\end{proof}
	
	The following analogy to the proof of the lattice version of Stone-Weierstrass theorem will be used multiple times in what follows.
	
	\begin{lemma}\label{lem:sepPoints}Let $M$ be a metric space and $D\subset \Lip_0(M)$ a subspace. Put $D':=D\cup\{\text{constant functions on $M$}\}\subset \Lip(M)$ and
		\[
		N(D):=\Big\{\bigvee_{i=1}^n\bigwedge_{j=1,\; i\neq j}^n f_{i,j}\colon \{f_{i,j}\colon i,j=1,\ldots,n,\; i\neq j\}\subset D',\; n\in\Nat\Big\}.
		\]
		Then the following holds for $E(D):=\Span (N(D)\cap \Lip_0(M))$.
		\begin{itemize}
			\item If $D$ separates the points of $M$, then $E(D)$ is dense in $(\Lip_0(M),\tau_p)$.
			\item If $D$ separates the points of $M$ uniformly with constant $C$, then $B_{\Lip_0(M)}\subset \overline{CB_{\Lip_0(M)}\cap E(D)}^{\tau_p}$.
		\end{itemize}
	\end{lemma}
	\begin{proof}Pick $f\in \Lip_0(M)$ and its basic $\tau_p$-open neighborhood given by a finite set $A\subset M$ and $\varepsilon>0$, where we may without loss of generality assume that $0\in A$. We \emph{claim} there exists $g\in E(D)$ (with $\|g\|_{\Lip}\leq (C+\varepsilon)\|f\|_{\Lip}$) such that $g|_A = f|_A$.
		
		Indeed, since $D$ separates the points (uniformly with constant $C\geq 1$), for any $x,y\in M$ there exists $f_{x,y}\in D$ such that $|f_{x,y}(x)-f_{x,y}(y)| = d(x,y)$ (and $\|f_{x,y}\|_{\Lip}\leq C+\varepsilon$). Multiplying $f_{x,y}$ by $\pm \tfrac{|f(x)-f(y)|}{d(x,y)}$ we obtain $f'_{x,y}\in D$ satisfying $f'_{x,y}(x)-f'_{x,y}(y) = f(x)-f(y)$ and adding a constant $f(y) - f'_{x,y}(y)$ we obtain $g_{x,y}\in D'$ such that $g_{x,y}(x)=f(x)$ and $g_{x,y}(y) = f(y)$ (and $\|g_{x,y}\|_{\Lip}\leq (C+\varepsilon)\|f\|_L$). Finally, we consider the function
		\[
		g:=\bigvee_{x\in A}\bigwedge_{y\in A,x\neq y} g_{x,y}.
		\]
		Then it is easy to check that $g|_A = f|_A$, $g\in E(D)$ (and $\|g\|_{\Lip}\leq (C+\varepsilon)\|f\|_L$). This finishes the proof of the claim above, and completes the proof that $E(D)$ is dense in $(\Lip_0(M),\tau_p)$.
		
		For the second part, we observe that for the function $h:=g\tfrac{C}{C+\varepsilon}$ we have that $\|h\|\leq C$ and $\|h|_A-f|_A\|_\infty \leq \|f|_A\|_\infty\cdot |\tfrac{C}{C+\varepsilon}-1|\to 0$, $\varepsilon\to 0$ and therefore, any $\tau_p$-neighborhood of the function $f$ intersects $CB_{\Lip_0(M)}\cap E(D)$ which implies that $f\in \overline{CB_{\Lip_0(M)}\cap E(D)}^{\tau_p}$.
	\end{proof}
	
	First consequence of Lemma~\ref{lem:sepPoints} is the following, which is the crucial step towards the proof of Theorem~\ref{thm:ball}.
	
	\begin{prop}\label{prop:LipDual}Let $X$ be a Banach space, $X\neq \{0\}$. Then 
		\[\dens \big(B_{\Lip_0(X^{*})},w^*\big)\leq \dens X.\] 
	\end{prop}
	\begin{proof}Let $D\subset B_X$ be dense set. Given $x^*,y^*\in X^*$ and rational $\delta>0$ there is $x\in D$ with $|(x^*-y^*)(x)|\geq (1-\delta)\|x^*-y^*\|$. But then for $r=\frac{\|x^*-y^*\|}{|(x^*-y^*)(x)|}$ we obtain that $f=rx\in \Lip_0(X^*)$ satisfies that $|f(x^*)-f(y^*)|=\|x^*-y^*\|$ and $\|f\|\leq \tfrac{1}{1-\delta}$. Thus, $\Span D\subset \Lip_0(X^*)$ separates points uniformly with constant $1$ and by Lemma~\ref{lem:sepPoints},
		\[
		B_{\Lip_0(X^*)}\subset \overline{\Span N(\Span D)\cap B_{\Lip_0(X^*)}}^{w^*}\subset \overline{\Span N(\Span_\Rat D)\cap B_{\Lip_0(X^*)}}^{w^*}.
		\]
		Since $|N(\Span_\Rat D)| = |D|$, this finishes the proof.
	\end{proof}
	
	Now, we are ready for the proofs of Proposition~\ref{prop:dualUnitBall} and Theorem~\ref{thm:ball}.
	
	\begin{proof}[Proof of Proposition~\ref{prop:dualUnitBall}]Equivalence of \ref{it:Fisom} and \ref{it:ballSep} is Fact~\ref{fact:denseInDualUnitBall}. If \ref{it:isom} holds, then $\F(M)$ isometrically embeds into $\F(\ell_\infty(\kappa))$ which in turn isometrically embeds into $\ell_\infty(\kappa)$ due to Fact~\ref{fact:denseInDualUnitBall} and Proposition~\ref{prop:LipDual} applied to $X = \ell_1(\kappa)$. Finally, if \ref{it:Fisom} holds pick linear isometry $T:\F(M)\to\ell_\infty(\kappa)$ and observe that then $T\circ \delta:M\to \ell_\infty(\kappa)$ is isometric embedding so \ref{it:isom} holds.
	\end{proof}
	
	\begin{proof}[Proof of Theorem~\ref{thm:ball}]
		Let $\kappa = \dens \big(B_{X^*},w^*\big)$. By Fact~\ref{fact:denseInDualUnitBall} there is a linear isometry $T:X\to \ell_\infty(\kappa)$ so by Proposition~\ref{prop:dualUnitBall} we obtain that $\dens \big(B_{\Lip_0(X)},w^*\big)\leq \kappa$.
		
		The ``In particular'' part follows from the well-known fact that in general we have $\dens \big(B_{X^*},w^*\big)\leq \dens X$.
	\end{proof}
	
	Finally, let us present the proof of Theorem~\ref{thm:tauPMain}.
	
	\begin{proof}[Proof of Theorem~\ref{thm:tauPMain}]
		By Lemma~\ref{lem:sepPoints}, it suffices to prove that there are countably many functions in $\Lip_0(M)$ separating the points of $M$.
		
		Using paracompactness of $M$ (see \cite[Theorem 5.1.3]{Eng}) we find for each $n$ an open cover $\A_n$ of $M$ which is locally finite and each $A\in\A_n$ is contained in an open ball of diameter at most $\tfrac{1}{n}$. Since there are at most $2^\omega$ points in $M$, we may without loss of generality assume that $|\A_n|\leq 2^\omega$, $n\in\Nat$ and denote its elements as $\A_n = \{U_x^n\colon x\in 2^\omega\}$.
		
		Pick a basis $\mathcal{B}$ of clopen sets in the Cantor set $2^\omega$ and put 
		\[
		\mathcal{N}:=\Big\{\bigcup_{i=1}^n B_i\colon B_1,\ldots,B_n\text{ are sets from }\mathcal{B}\Big\}.
		\]
		Then $\mathcal{N}$ is a countable family of sets satisfying that for any two disjoint finite sets $S,S'\subset 2^\omega$ we find $N,N'\in\mathcal{N}$ with $S\subset N\setminus N'$ and $S'\subset N'\setminus N$. Finally, for each $N\in\mathcal{N}$ we put 
		\[
		F_{n,N}:=\overline{\bigcup\{U_x^n\colon x\in N\}}.
		\]
		and consider the countable family of $1$-Lipschitz functions
		\[\mathcal{R}:=\{f_{n,N},\colon n\in\Nat, N\in\mathcal{N}\},\]
		where each $f_{n,N}$ is defined as $f_{n,N}(a):=d(a,F_{n,N}) - d(0,F_{n,N})$. We shall verify this family separates the points of $M$.
		
		Pick two distinct points $a,b\in M$ and find $n\in\Nat$ with $B(a,\tfrac{2}{n})\cap B(b,\tfrac{2}{n}) = \emptyset$. Since $\A_n$ is locally finite covering, there is $m\in\Nat$, $m\geq n$ such that there are only finitely many sets from $\A_n$ intersecting $B(a,\tfrac{1}{m})$ and finitely many sets from $\A_n$ intersecting $B(b,\tfrac{1}{m})$. Thus, there are finite sets $S,S'\subset 2^\omega$ such that $\{U_{x}^n\colon x\in S\}$ and $\{U_{x}^n\colon x\in S'\}$ are the sets from $\A_n$ intersecting $B(a,\tfrac{1}{m})$ and $B(b,\tfrac{1}{m})$, respectively. Since each of the sets $U_x^n$, $x\in 2^\omega$ has diameter less or equal to $\tfrac{1}{n}$, we have that $\bigcup_{x\in S}U_{x}^n\subset B(a,\tfrac{2}{n})$ and $\bigcup_{x\in S'}U_{x}^n\subset B(b,\tfrac{2}{n})$, so $S\cap S' = \emptyset$. By the choice of $\mathcal{N}$ we pick $N,N'\in\mathcal{N}$ with $S\subset N\setminus N'$ and $S'\subset N'\setminus N$. Then $a\in F_{n,N}$ and $b\in F_{n,N'}$. On the other hand, since $B(a,\tfrac{1}{m})\cap U_x^n = \emptyset$ for $x\in 2^\omega\setminus S$, we have $B(a,\tfrac{1}{m})\cap \big(\bigcup\{U_x^n\colon x\in N'\}\big) = \emptyset$ and therefore $d(a,F_{n,N'})\geq \tfrac{1}{m} >0$. Thus, we have $f_{n,N'}(a) > - d(0,F_{n,N'}) = f_{n,N'}(b)$. Since $a,b\in M$ were arbitrary, this proves that $\mathcal{R}$ separates the points of $M$.
	\end{proof}

	\section{Locally separable complete metric spaces}\label{sec:locSep}
	
	In this section, let $M$ be a metric space with $\dens M\leq 2^\omega$. Recall that, using the Hahn-Banach theorem, $w^*$-separability of $\Lip_0(M)$ is equivalent to the existence of countably many functions separating the points of $\F(M)$. By Theorem~\ref{thm:tauPMain}, there are countably many Lipschitz functions separating the points of $M$. Using \cite[Proposition 2.11]{AGP24}, this implies that then there are countable many functions in $\Lip_0(M)$ separating the points of the dense subspace $\Span \delta(M)\subset \F(M)$. However, it is not clear to the authors whether those functions separate also all the points in $\F(M)$. Thus, in order to obtain $w^*$-separability, it seems something more is needed. It was observed in \cite[text below the proof of Proposition 2.10]{AGP24} that supposing $M$ is uniformly discrete, we obtain $w^*$-separability of $\Lip_0(M)$. Trying to analyze the argument, we arrive at the following sufficient condition, which is inspired by the proof of \cite[Proposition 3.5]{GPP23}.
	
	\begin{prop}\label{prop:sufficient}
		Let $M$ be a complete metric space with $\dens M\leq 2^\omega$. Suppose there exists a countable sequence of 1-Lipschitz functions $(f_n)_{n\in\Nat}$ such that
		\begin{enumerate}[label=(C\alph*)]        
			\item\label{cond:countable} for every $x\in M$ there exists $\varepsilon_0>0$ such that for every $\varepsilon\in (0,\varepsilon_0)$ and $g\in \Lip_0(M)$ with $\suppt g\subset B(x,\varepsilon)$ there exists $K>0$ satisfying
			\[
			|g(y)-g(z)|\leq K\Big(\sup_{n\in\Nat} |f_n(y)-f_n(z)|\Big),\qquad y,z\in M.
			\]
		\end{enumerate}
		Then $(\Lip_0(M),w^*)$ is separable.
	\end{prop}
	
	The main result of this section is that it is possible to satisfy this condition in locally separable complete metric spaces.
	
	\begin{thm}\label{thm:locSep}
		Let $M$ be a complete locally separable metric space with $\dens M\leq 2^\omega$. Then $(\Lip_0(M),w^*)$ is separable.
	\end{thm}
	
	One might wonder whether it is the case that even the dual unit ball of $\Lip_0(M)$ is $w^*$-separable. We note this is not the case since by \cite[Example 2.12]{AGP24} there exists a uniformly discrete metric space of density at most continuum which does not Lipschitz embed into $\ell_\infty$ (and therefore, using Proposition~\ref{prop:dualUnitBall} the dual unit ball $(B_{\Lip_0(M)},w^*)$ is not separable).
	
	The remainder of this section is mostly devoted to the proofs of Proposition~\ref{prop:sufficient} and Theorem~\ref{thm:locSep}.
	
	The following is a useful tool to recognize whether a given point in a Lipschitz-free space is zero. It is a rather straightforward consequence of the results from \cite{APPP20} concerning the notion of a support.
	
	\begin{lemma}\label{lem:notZero}
		Let $M$ be a complete metric space and $\mu\in \F(M)$. Then $\mu\neq 0$ if and only if there exists $x\in M$ such that for every $\varepsilon>0$ there is $f\in \Lip_0(M)$ with $\suppt f\subset B(x,\varepsilon)$ and $f(\mu)\neq 0$.
	\end{lemma}
	\begin{proof}
		If $\mu\neq 0$, by the Intersection theorem \cite[Theorem 2.1]{APPP20} there exists $x\in \suppt \mu$ and conversely, we have $\suppt 0 = \emptyset$. Thus, we obtain that $\mu\neq 0$ if and only if $\suppt \mu\neq 0$ and now it suffices to apply \cite[Proposition 2.7]{APPP20}. 
	\end{proof}
	
	\begin{proof}[Proof of Proposition~\ref{prop:sufficient}]We may without loss of generality assume that $f_n(0)=0$, $n\in\Nat$. Let us consider the function $f:M\to \ell_\infty$ given as $f(x):=\big( f_n(x)\big)_{n\in\Nat}$. Since functions $f_n$ are 1-Lipschitz, we have that $f$ is a well-defined and $1$-Lipschitz function. Thus, by the universal property of Lipschitz-free spaces, it extends uniquely to a norm-one linear operator $T:\F(M)\to \F(\overline{\rng f})\subset \F(\ell_\infty)$ satisfying $T\circ \delta = \delta\circ f$. It suffices to prove $T$ is one-to-one (because then $T^*$ has $w^*$-dense range and since $\F(\ell_\infty)$ has $w^*$-separable dual by Proposition~\ref{prop:dualUnitBall}, we would obtain that $\F(M)$ has $w^*$-separable dual as well).
		
		Pick $0\neq \mu\in \F(M)$. Let $x\in M$ be as in Lemma~\ref{lem:notZero} and pick the corresponding $\varepsilon_0>0$ from the validity of \ref{cond:countable}. We shall check that $f(x)\in \overline{\rng f}$ satisfies the condition from Lemma~\ref{lem:notZero} for $T(\mu)\in \F(\overline{\rng f})$, which will imply that $T(\mu)\neq 0$. Pick $\varepsilon > 0$. Using the continuity of $f$ we pick $\delta\in (0,\varepsilon_0)$ satisfying that $f(B(x,\delta))\subset B(f(x),\tfrac{\varepsilon}{2})$. Using the choice of $x$ we find $g\in\Lip_0(M)$ with $\suppt g\subset B(x,\delta)$ and $g(\mu)\neq 0$. Put $\varphi:=g\circ f^{-1}$. Then using \ref{cond:countable} there exists $K>0$ such that for any $y,z\in M$  we have
		\[
		|\varphi(f(y)) - \varphi(f(z))|\leq K \|f(y)-f(z)\|_\infty,
		\]
		so the function $\varphi$ is Lipschitz on $\rng f$ and therefore we may extend it to a Lipschitz function defined on $\overline{\rng f}$ (denoting this extension again by $\varphi$). 
		
		Further, we have $\suppt \varphi\subset B(f(x),\varepsilon)$. Indeed, pick $z\in M$. If $f(z)\notin B(f(x),\tfrac{\varepsilon}{2})\supset f(B(x,\delta))$, then $z\notin B(x,\delta)\supset \suppt g$, so we have $g(z)=0$ which implies that $\varphi(f(z))=0$; thus, $\{y\in \rng f\colon \varphi(y)\neq 0\}\subset B(f(x),\tfrac{\varepsilon}{2})$ and since $\{y\in \rng f\colon \varphi(y)\neq 0\}$ is a dense subset of $\suppt \varphi$, this finishes the proof that $\suppt \varphi\subset B(f(x),\varepsilon)$.
		
		Finally, since $T\circ \delta = \delta\circ f$, we have $(\varphi\circ T)(\delta(v)) = \varphi(f(v)) = g(v) = g(\delta(v))$ for every $v\in M$, so $\varphi(T(\mu)) = g(\mu)\neq 0$ and therefore, by Lemma~\ref{lem:notZero} we obtain that $\varphi\in \Lip_0(\overline{\rng f})$ witnesses the fact that $T(\mu)\neq 0$.
	\end{proof}
	
	\begin{proof}[Proof of Theorem~\ref{thm:locSep}]
		Recall that every metric space admits a $\sigma$-discrete base, see e.g. \cite[Theorem 4.4.3]{Eng}.
		Thus, pick $\sigma$-discrete base of $M$ denoted by $\mathcal B = \bigcup_{k\in\Nat}\mathcal B_k$, where each $\B_k$ consists of open sets and is discrete (that is, for any $k\in\Nat$ and $x\in M$ there is an open neighborhood of $x$ which intersects at most one set from $\B_k$).
		Since $M$ is locally separable, we may additionally assume that every $B\in\mathcal B$ is separable and $\diam B\leq 1$
		and we fix maps $\varphi_B:\Nat \to B$ whose range is dense in $B$.
		
		Since each $\mathcal B_k$ of size at most continuum we may find $\mathcal B_{k,l}\subseteq \mathcal B_k$, $l\in\Nat$, such that for every pair of distinct sets $B, C\in \mathcal B_k$ there is some $l\in\Nat$ satisfying $B\in\mathcal B_{k,l}$ and $C\notin \mathcal B_{k,l}$ (one can for example enumerate $\B_k = \{B_s\colon s\in 2^\omega\}$ and then consider $\B_{k,l,j}:=\{B_s\colon s(l)=j\}$ for each $l\in\Nat$ and $j\in\{0,1\}$). 
		
		Finally, consider the countable family of functions $\{f_{k,l,m}\colon k,l,m\in\Nat\}\cup\{g_{k,l}\colon k,l\in\Nat\}$ given as 
		\[f_{k, l, m}(x) := \dist\Big(x, \{\varphi_{B}(m)\colon B\in \B_{k, l}\}\cup (M\setminus \bigcup \B_{k, l})\Big),\qquad x\in M,\]
		\[g_{k,l}(x):=\dist\Big(x, M\setminus \bigcup \B_{k, l}\Big),\qquad x\in M.\]
		Since each $f_{k,l,m}$ and $g_{k,l}$ is obviously $1$-Lipschitz, it remains to check that the collection $\{f_{k,l,m}\colon k,l,m\in\Nat\}\cup\{g_{k,l}\colon k,l\in\Nat\}$ satisfies condition \ref{cond:countable} (then, by Proposition~\ref{prop:sufficient} we obtain that $\Lip_0(M)$ is $w^*$-separable and we are done).
		
		\textbf{Step 1:} We fix $x\in M$ which is not isolated and check for this particular $x$ validity of condition \ref{cond:countable} (we shall see in the proof that in this Step we are using functions $f_{k,l,m}$).
		
		We can find $k\in\Nat$ and $B\in \B_k$ for which $x\in B$. Since $x$ is not isolated, there is $\varepsilon\in (0,\tfrac{1}{9})$ such that $B(x,10\varepsilon)\subseteq B$ and $B\setminus \overline{B}(x,2\varepsilon)\neq \emptyset$ (where $\overline{B}(x,2\varepsilon)$ denotes the closed ball).
		Suppose now that we are given $g\in\Lip(M)$ with $\suppt g\subseteq B(x,\varepsilon)$. We may without loss of generality assume $g$ is $1$-Lipschitz. Put 
		\[
		K:=\max\Big\{\frac{2}{9\varepsilon}, 2 + \frac{\dist(x, M\setminus B(x,\varepsilon))}{\varepsilon}\Big\}
		\]
		and pick $y,z\in M$, without loss of generality we assume that $y\in B(x,\varepsilon)$. Note that for $l\in\Nat$ with $B\in \B_{k,l}$ we have
		\[
		\{\varphi_{B'}(m)\colon B\neq B'\in \B_{k,l}\}\cup (M\setminus \bigcup \B_{k,l})\quad \subset \quad M\setminus B \quad \subset\quad  M\setminus B(x,10\varepsilon),
		\]
		and since for any $w\notin B(x,10\varepsilon)$ we have $d(w,y)\geq d(w,x) - d(x,y)\geq 9\varepsilon$, this proves that
		\begin{equation}\label{eq:fValue}
			f_{k,l,m}(y)\geq \min \{9\varepsilon, d(y,\varphi_B(m))\},\quad \text{whenever $l\in\Nat$ is such that $B\in\B_{k,l}$.}
		\end{equation}
		Now, we distinguish two cases.
		
		\textbf{Case 1:} Assume that $z\in B$. Pick $l\in\Nat$ with $B\in \B_{k,l}$ and, since $z\in B$, we find $m\in\omega$ such that $d(\varphi_B(m),z)<\frac{9\varepsilon}{2+18\varepsilon}d(y,z)$. Then using \eqref{eq:fValue} and $\diam B\leq 1$, we obtain $f_{k,l,m}(y) \geq \min\{1,9\varepsilon\}d(y,\varphi_B(m)) = 9\varepsilon d(y,\varphi_B(m))$ and of course, $f_{k,l,m}(z)\leq d(z,\varphi_B(m))$. Thus,
		\[\begin{split}
			K|f_{k,l,m}(y) - f_{k,l,m}(z)| & \geq K\big(9\varepsilon d(y,\varphi_B(m)) - d(z,\varphi_B(m))\Big) \\ 
			& \geq K\Big(9\varepsilon d(y,z) - (1+9\varepsilon)d(z,\varphi_B(m))\Big)\\
			& \geq \frac{K 9\varepsilon}{2} d(y,z) \geq \frac{K 9\varepsilon}{2}|g(y)-g(z)| \geq |g(y)-g(z)|,
		\end{split}\]
		which verifies the condition \ref{cond:countable} in Case 1.
		
		\medskip
		
		\textbf{Case 2:} Assume that $z\notin B$. Notice that then we can find $l\in\Nat$ such that $z\notin \bigcup \B_{k,l}$. Indeed, this is obvious if $z\notin \bigcup \B_k$ and if $z\in C\in \B_k$ with $C\neq B$, then we just pick $l\in\Nat$ satisfying $B\in \B_{k,l}$ and $C\notin \B_{k,l}$ which due to the fact that $\B_k$ is discrete implies that $z\notin \bigcup \B_{k,l}$.
		
		Since $B\setminus \overline{B}(x,2\varepsilon)\neq \emptyset$, there exists $\varphi_B(m)$ with $d(\varphi_B(m),x) > 2\varepsilon$ and therefore $d(\varphi_B(m),y) > \varepsilon$, which using \eqref{eq:fValue} implies that $f_{k,l,m}(y) > \varepsilon$. Further, since $z\in M\setminus \bigcup \B_{k,l}$ we have $f_{k,l,m}(z) = 0$. Thus, we obtain
		\[\begin{split}
			|g(y)-g(z)| & = |g(y)| \leq \varepsilon + |g(x)|\leq \varepsilon + \dist\big(x, M\setminus B(x,\varepsilon)\big)\\
			& \leq K\varepsilon < K|f_{k,l,m}(z) - f_{k,l,m}(y)|,
		\end{split}\]
		which verifies condition \ref{cond:countable} in Case 2.
		
		\medskip
		
		\textbf{Step 2:} We fix $x\in M$ which is isolated and check for this particular $x$ validity of condition \ref{cond:countable} (we shall see in the proof that in this Step we are using functions $g_{k,l}$).
		
		We can find $k\in\Nat$ and $B\in \B_k$ for which $\{x\} = B$. Further, pick $\varepsilon>0$ with $\varepsilon < d(x,M\setminus\{x\})$ and $g\in\Lip_0(M)$ with $\suppt g\subset B(x,\varepsilon) = \{x\}$. We may without loss of generality assume $g$ is $1$-Lipschitz. Put $K:=1$ and fix $y,z\in M$. We may without loss of generality assume that $y=x$. Notice that we can find $l\in\Nat$ such that $z\notin \bigcup \B_{k,l}$. Indeed, this is obvious if $z\notin \bigcup \B_k$ and if $z\in C\in \B_k$ with $C\neq \{x\} = B$, then we just pick $l\in\Nat$ satisfying $B\in\B_{k,l}$ and $C\notin \B_{k,l}$ which due to the fact that $\B_k$ is discrete implies that $z\notin \bigcup \B_{k,l}$. Thus, we obtain $g_{k,l}(x) \geq d(x,M\setminus\{x\})$ and $g_{k,l}(z) = 0$ and therefore,
		\[
		|g(y)-g(z)|\leq d(x,z)\leq d(x,M\setminus\{x\})\leq g_{k,l}(x) - g_{k,l}(z) = |g_{k,l}(x) - g_{k,l}(z)|,
		\]
		which verifies condition \ref{cond:countable}.
	\end{proof}
	
	Let us note the that Theorem~\ref{thm:locSep} seems to be actually quite close to the general answer as we have the following (which implies that if we were able to prove Theorem~\ref{thm:locSep} without the assumption that $M$ is complete, we would have a positive answer to the whole Conjecture~\ref{conj:main}).
	
	\begin{fact}\label{fact:discreteSubset}Any metric space $M$ embeds isometrically into completion of a discrete metric space of the cardinality equal to $\dens M$.\\
		(Note: this discrete space is even a countable union of uniformly discrete spaces)
	\end{fact}
	\begin{proof}For each $n\in\mathbb{N}$ pick a maximal $\tfrac{1}{n}$-separated subset of $M$ and call it $A_n$. Then put $Y:=\{(x,\tfrac{1}{n})\colon x\in A_n,\; n\in\Nat\}\subset M\oplus_1 [0,1]$ (that is, the metric on $M\times [0,1]$ is given by $d((x,t),(x',t')) = d(x,x') + |t-t'|$). Of course, $M$ is isometric to $\tilde{M} = M\times \{0\}\subset M\oplus_1 [0,1]$ so it suffices to check that $\tilde{M}\subset \overline{Y}$ and that $Y$ is discrete.
		
		Pick $x\in M$, $n\in \Nat$ and using the maximality of $A_n$ find $a_n\in A_n$ with $d(x,a_n)\leq \tfrac{1}{n}$. Then $d((x,0),(a_n,\tfrac{1}{n}))\leq \tfrac{2}{n}$, so we have $(a_n,\tfrac{1}{n})\to (x,0)$ and therefore $\tilde{M}\subset \overline{Y}$.
		
		Further, the set $Y$ is discrete, because for $n,m\in\Nat$, $a\in A_n$ and $a'\in A_m$ the distance between $(a,\tfrac{1}{n})$ and $(a',\tfrac{1}{m})$ is at least $\tfrac{1}{n}$ if $n=m$ and at least $\min\{\tfrac{1}{n-1}-\tfrac{1}{n},\tfrac{1}{n}-\tfrac{1}{n+1}\}$ if $n\neq m$.\\
		(Note: $Y$ is even a countable union of uniformly discrete sets.)
	\end{proof}
	
	\section{Further possible reductions}\label{sec:banach}
	
	The goal of this section is to derive new characterizations for the $w^*$-separability of spaces of Lipschitz functions and their unit balls. We begin with the following auxiliary proposition. It might be seen as the extension of the known fact that $\ell_1(2^\omega)$ embeds isometrically into $\ell_\infty$ (see e.g. \cite[Fact 7.26]{HMVZ}).
	
	\begin{prop}\label{prop:ell1Injection}
		There is a linear isometric embedding of $\left(\bigoplus_{2^{\omega}}\ell_\infty\right)_{\ell_1}$ into $\ell_\infty$.
	\end{prop}
	\begin{proof}
		Let us denote $X=\left(\bigoplus_{2^{\omega}}\ell_\infty\right)_{\ell_1}$ and let $K=\{0,1\}^{\Nat}$ be the Cantor cube. Consider the set $\varGamma$ of all continuous functions of the form $f:K\to \{1,2,\ldots,r\}$ for some $r\in \Nat$. We notice that $\varGamma$ is countable since there are only countably many clopen subsets of $K$. If $\varGamma=\{f_n:n \in \Nat\}$ we consider
		\[\varLambda=\{(n,u,s): n \in \Nat, \ u:\rng(f_n) \to \Nat, \  s:\rng(f_n) \to \{-1,1\} \}.\] 
		For each $(n,u,s)\in \varLambda$ we consider the functional $\xi_{(n,u,s)}:X\to \Rea$ defined, for every $x=(x(t))_{t\in K}$, as follows:
		\[\xi_{(n,u,s)}(x)=\sum_{j\in \rng(f_n)}\sum_{t\in f_n^{-1}[\{j\}]} s(j)\cdot x(t)(u(j)).\]
		It is readily seen that $\xi_{(n,u,s)}\in X^*$ and $\|\xi_{(n,u,s)}\|\leq 1$, for each $(n,u,s) \in \varLambda$.
		Now we claim that $\{\xi_{(n,u,s)}:(n,u,s)\in \varLambda\}$ is a $1$-norming subset of $B_{X^*}$. Indeed, let $x=(x(t))_{t\in K}\in X$ and $\epsilon>0$ be arbitrary.  Let $U\subset K$ be a finite set such that 
		$\sum_{t \in U}\|x(t)\|_{\ell_\infty}>\|x\|-\frac{\epsilon}{4}$.
		Assuming that $|U|=m$, for each $t \in U$, let $n_t$ be such that $|x(t)(n_t)|>\|x(t)\|_{\ell_\infty}-\frac{\epsilon}{2m}$. Let $u:\{1,2,\ldots,m\}\to \{n_t:t\in U\}$ be a surjective function and $s:\{1,2,\ldots,m\}\to \{-1,1\}$ be given by $s(j)=\sgn(x(t)(u(j)))$, where $t$ is such that $n_t = u(j)$.
		Finally we fix $n \in \Nat$ such that $|\rng(f_n)|= m$ and, for each $1\leq j \leq m$ and $t\in U$ with $n_t=u(j)$, we have $t\in f_n^{-1}[\{j\}]$. Then
		\begin{align*}
			|\xi_{(n,u,s)}(x)|& \geq \sum_{t\in U}|x(t)(n_t)|-\sum_{t\in K\setminus U}\|x(t)\|_{\ell_\infty} \geq \Big(\sum_{t\in U}\|x(t)\|_{\ell_\infty}-\frac{\epsilon}{2}\Big)-\sum_{t\in K\setminus U}\|x(t)\|_{\ell_\infty}\\
			&\geq \|x\|-2\Big(\sum_{t\in K\setminus U}\|x(t)\|_{\ell_\infty}\Big)-\frac{\epsilon}{2}\geq \|x\|-\epsilon.
		\end{align*}
		Thus, $X\ni x\mapsto (\xi_{n,u,s}(x))_{(n,u,s)\in\Lambda}\in \ell_\infty(\Lambda)$ is linear isometry, which finishes the proof because $\Lambda$ is a countable set.
	\end{proof}
	
	As an almost immediate consequence we obtain the following.
	
	\begin{cor}\label{cor:preservation}Let $X_i$, $i<2^\omega$ be Banach spaces with $w^*$-separable duals (resp. $w^*$-separable dual unit balls). Assume there exists an injective bounded linear operator (resp. linear isometry) from a Banach space $Y$ into $\Big(\bigoplus_{2^\omega} X_i\Big)_{\ell_1}$. Then $Y$ has $w^*$-separable dual (resp. $w^*$-separable dual unit ball) as well.
	\end{cor}
	\begin{proof}It suffices to prove that there is a linear injection (resp. isometry) from $\Big(\bigoplus_{2^\omega} X_i\Big)_{\ell_1}$ into $\ell_\infty$ (see e.g. \cite[Fact 4.10]{HMVZ}, resp. Fact~\ref{fact:denseInDualUnitBall}). By the assumption, there are norm-one injective linear mappings (resp. isometries) $T_i:X_i\to \ell_\infty$. Thus, the mapping $T:\Big(\bigoplus_{2^\omega} X_i\Big)_{\ell_1} \to \Big(\bigoplus_{2^\omega} \ell_\infty\Big)_{\ell_1}$ given by $T((x_i)_i):=(T_i(x_i))_i$ is bounded linear injection (resp. isometry) and so it suffices to use Proposition~\ref{prop:ell1Injection} to obtain a linear injection (resp. isometry) from $\Big(\bigoplus_{2^\omega} X_i\Big)_{\ell_1}$ into $\ell_\infty$.
	\end{proof}
	
	Another quite an immediate consequence is the following.
	
	\begin{cor}\label{cor:charactWeakStarSep}Let $X$ be a Banach space. Then the following conditions are equivalent.
		\begin{enumerate}[label=(\alph*)] 
			\item\label{it:weakStarSep} $(X^*,w^*)$ is separable
			\item\label{it:intoEll} There is a bounded linear injection of $X$ into $\ell_\infty$.
			\item\label{it:intoSumEll} There is a bounded linear injection of $X$ into $\left(\bigoplus_{2^{\omega}}\ell_\infty\right)_{\ell_1}$.
		\end{enumerate}
		By replacing $X^*$ with $B_{X^*}$ in condition~\ref{it:weakStarSep} and using a linear isometric embedding instead of a bounded linear injection in conditions \ref{it:intoEll} and \ref{it:intoSumEll}, we also obtain a characterization for the $w^*$-separability of $B_{X^*}$.
	\end{cor}
	\begin{proof}The equivalence between \ref{it:weakStarSep} and \ref{it:intoEll} is a well-known fact (see \cite[Fact 4.10]{HMVZ} and Fact~\ref{fact:denseInDualUnitBall}). The implication \ref{it:intoEll} $\Rightarrow$ \ref{it:intoSumEll} is evident and \ref{it:intoSumEll} $\Rightarrow$ \ref{it:intoEll} follows immediately from Proposition~\ref{prop:ell1Injection}.
	\end{proof}
	
	Now, we shall proceed towards the proof that Lipschitz-free spaces over $\ell_\infty$ and its $c_0(2^\omega)$-sum are linearly isomorphic, see Theorem~\ref{thm:c0Andellinfty}. The first step towards the proof is in \ref{Prop:C0ellinftyisanAbsoluteRetract}. We believe it is a part of a folklore knowledge, but since we did not find a reference, for the convenience of the reader we write the argument below. We emphasize the ``In particular'' part, since this is the one which we shall need in the proof of Theorem \ref{thm:c0Andellinfty}.
	
	\begin{prop}\label{Prop:C0ellinftyisanAbsoluteRetract}Given a set $I$, $\lambda>0$ and Banach spaces $X_i$, $i\in I$ which are absolute $\lambda$-Lipschitz retracts, the Banach spaces $\left(\bigoplus_{I} X_i\right)_{\ell_\infty}$ and $\left(\bigoplus_{I} X_i\right)_{c_0}$ are absolute $\lambda$-Lipschitz retract and absolute $2\lambda$-Lipschitz retract, respectively.
		
		In particular, $\left(\bigoplus_{I}\ell_\infty\right)_{c_0}$ is an absolute Lipschitz retract.
	\end{prop}
	\begin{proof}Denote $X = \left(\bigoplus_{I} X_i\right)_{\ell_\infty}$ and $Y = \left(\bigoplus_{I} X_i\right)_{c_0}$.
		Proving that $X$ is $\lambda$-Lipschitz retract is coordinate-wise. That is, we first pick a set $J$ such that for every $i\in I$ there exists an isometry $T_i:X_i\to \ell_\infty(J)$ and a $\lambda$-Lipschitz retraction $R_i:\ell_\infty(J)\to T_i(X_i)$. Then, we easily check that $T:X\to \left(\bigoplus_{I} \ell_\infty(J)\right)_{\ell_\infty}$ given by $T(x)(i):=T_i(x(i))$ is an isometry and $R:\left(\bigoplus_{I} \ell_\infty(J)\right)_{\ell_\infty}\to T(X)$ given by $R(x)(i):=R_i(x(i))$ is $\lambda$-Lipschitz retraction. Thus, $X$ is isometric to a $\lambda$-Lipschitz retract of $\left(\bigoplus_{I} \ell_\infty(J)\right)_{\ell_\infty} = \ell_\infty(I\times J)$, which is well-known to be an absolute $1$-Lipschitz retract.
		
		Thus, in order to prove $Y$ is absolute $2\lambda$-Lipschitz retract it suffices to show that $Y$ is $2$-Lipschitz retract of $X$. Let $d(x)=\dist(x,Y)$ for every $x\in X$. The required retraction is given similarly as in \cite[Example 1.5]{BeLiBook} by the formula
		\[
		R(x)_i:=\left\{
		\begin{array}{ll}
			0& \text{ if }d(x)>\|x_i\|_{X_i}\\
			\Big(1 - \frac{d(x)}{\|x_i\|_{X_i}}\Big)\cdot x_i &\text{ if }d(x)\leq \|x_i\|_{X_i}.
		\end{array} \right.
		\]
		Now, one easily checks this formula gives really a $2$-Lipschitz retraction. Since this is omitted even in \cite[Example 1.5]{BeLiBook}, for the convenience of the reader we give the details below.
		
		First, we claim that, for every $\epsilon>0$ and $x\in X$, if $x=(x_i)_i$, then the set $I_{\epsilon}(x)=\{i\in I\colon \|x_i\|_{X_i}\geq d(x)+\epsilon\}$ is finite. Indeed, given $\epsilon>0$ we let $a=(a_i)_i\in Y$ be such that $d(x)\leq \|x-a\|<d(x)+\epsilon$.
		If $I_{\epsilon}(x)$ is infinite, then for every $\delta>0$ there is $i\in I_{\epsilon}(x)$ such that $\|a_i\|_{X_i}<\delta$. Hence
		\[d(x)+\epsilon\leq \|x_i\|_{X_i}\leq \|x_i-a_i\|_{X_i}+\|a_i\|_{X_i}\leq \|x-a\|+\delta.\]
		Hence $d(x)+\epsilon\leq  \|x-a\|$ which is a contradiction. This proves the claim and shows that $R(x)\in Y$ for every $x\in X$.
		
		In order to check $R$ is Lipschitz, pick $x,y\in X$ and $i\in I$. We may without loss of generality assume that $d(x)\leq \|x_i\|_{X_i}$. If $d(y) > \|y_i\|$ then we obtain
		\[\begin{split}
			\|R(x)_i-R(y)_i\|&=\|x_i\|-d(x)= (\|x_i\|-\|y_i\|)+(\|y_i\|-d(y))+(d(y)-d(x))\\
			&\leq |\|x_i\|-\|y_i\||+ |d(y)-d(x)| \leq  2\|x-y\|.
		\end{split}\]
		On the other hand, if $d(y) \leq \|y_i\|$ then we have
		\[\begin{split}
			\|R(x)_i-R(y)_i\|&=\Big\|\frac{\|x_i\| - d(x)}{\|x_i\|}x_i - \frac{\|x_i\| - d(x)}{\|x_i\|}y_i\Big\|\\
			& \leq \Big(\frac{\|x_i\|-d(x)}{\|x_i\|}\Big)\|x_i-y_i\| + \|y_i\|\cdot\Big|\frac{\|x_i\| - d(x)}{\|x_i\|} - \frac{\|y_i\| - d(y)}{\|y_i\|}\Big|\\
			& = \Big(1 - \frac{d(x)}{\|x_i\|}\Big)\|x_i-y_i\| + \Big|d(y) - d(x)\frac{\|y_i\|}{\|x_i\|}\Big|\\
			& \leq \Big(1 - \frac{d(x)}{\|x_i\|}\Big)\|x_i-y_i\| + |d(y)-d(x)| + \frac{d(x)}{\|x_i\|}|\|x_i\| - \|y_i\||\\
			& \leq 2\|x_i-y_i\| + \frac{d(x)}{\|x_i\|}(|\|x_i\| - \|y_i\|| - \|x_i-y_i\|)\leq 2\|x_i-y_i\|.
		\end{split}\]
	\end{proof}
	
	The next result is a slight modification of \cite[Proposition 5.4]{Kal11}.
	
	\begin{prop}\label{Prop:EmbeddingC0(ellinfty)intoellinfty}
		There is a $2$-Lipschitz embedding of $(\bigoplus_{2^{\omega}}\ell_\infty)_{c_0}$ into $\ell_\infty$.
	\end{prop}
	\begin{proof}
		Let $X=(\bigoplus_{\Rea}\ell_\infty)_{c_0}$. An arbitrary element $x\in X$ will be denoted $x=(x_t)_{t \in \Rea}$ 
		where $x_t=(x^t_n)_{n\in \Nat}\in \ell_\infty$ for every $t \in \Rea$. 
		For each $s \in \Rea$, by $\iota_{s}:\ell_\infty\to X$ we denote the canonical inclusion in the $s$-coordinate, that is, $\iota_s(u)=(x_t)_t$ where $x_s=u$ and $x_t=0$ whenever $s \neq t$.
		
		Let $\varLambda$ be the set of all tuples $(a,b,c,m)$ where $a\in \{-1,1\}$,\ $b,c\in \Rat$ are such that $b<c$, and $m\in \Nat$. It is evident that $\varLambda$ is a countable set. For each  $(a,b,c,m)\in \varLambda$ we consider the function
		$f_{a,b,c,m}:X \to  \Rea$ given by
		\[f_{a,b,c,m}(x)=\sup\left\{\max\{ax^{t}_{m},0\}: t \in (b,c)\right\}.\]
		It is easily seen that every function $f_{a,b,c,m}$ is an element of $B_{\Lip_0(X)}$. Therefore, the formula $F(x) = (f_{a,b,c,m}(x))_{(a,b,c,m)\in  \varLambda}$ defines a $1$-Lipschitz mapping from $X$ to $\ell_\infty(\varLambda)$. Now, we shall prove that $F$ is a Lipschitz embedding. Let $x, y \in  X$ and $\epsilon > 0$ be arbitrary and let 
		$s_0\in \Rea$, $m_0\in \Nat$ and $a\in \{-1,1\}$ be such that 
		\[a(x^{s_0}_{m_0}-y^{s_0}_{m_0})\geq \|x-y\|-\frac{\epsilon}{3}.\]
		
		Then we pick a finite subset $\{t_1 , \ldots , t_n \}$ containing $s_0$ and so that  
		\[\max\left\{\left\|x -\sum_{j=1}^n \iota_{t_j}(x_{t_j})\right\|,  \left\|y -\sum_{j=1}^n \iota_{t_j}(y_{t_j})\right\|\right\}<\frac{\epsilon}{6}.\]
		Additionally, we fix rationals $b$, $c$ such that $(b,c)\cap \{t_1,\ldots,t_n\}=\{s_0\}$. Since $f_{a,b,c,m_0}$ and $f_{-a,b,c,m_0}$ are elements of $B_{\Lip_0(X)}$ we have
		\[|f_{a,b,c,m_0}(x) - \max\{ax^{s_0}_{m_0},0\}| 
		= \Big|f_{a,b,c,m_0}(x) - f_{a,b,c,m_0}\Big(\sum_{j=1}^n \iota_{t_j}(x_{t_j})\Big)\Big|< \frac{\epsilon}{6}\]
		and
		\[|f_{-a,b,c,m_0}(x) -\max\{-ax^{s_0}_{m_0},0\} | = \Big|f_{-a,b,c,m_0}(x) - f_{-a,b,c,m_0}\Big(\sum_{j=1}^n \iota_{t_j}(x_{t_j})\Big)\Big|< \frac{\epsilon}{6}.\]
		Hence,
		\begin{align*}
			|f_{a,b,c,m_0}(x) - f_{-a,b,c,m_0} (x) -ax^{s_0}_{m_0}|&= |f_{a,b,c,m_0}(x) - f_{-a,b,c,m_0} (x) \\
			&-(\max\{ax^{s_0}_{m_0},0\}-\max\{-ax^{s_0}_{m_0},0\})|< \frac{\epsilon}{3}.
		\end{align*}
		In a similar manner, we obtain
		\[|f_{a,b,c,m_0}(y) - f_{-a,b,c,m_0} (y) -ay^{s_0}_{m_0}| < \frac{\epsilon}{3}.\]
		Hence
		\begin{align*}
			f_{a,b,c,m_0}(x) - f_{a,b,c,m_0}(y) - f_{-a,b,c,m_0} (x) + f_{-a,b,c,m_0} (y) & >a(x^{s_0}_{m_0}-y^{s_0}_{m_0})-\frac{2\epsilon}{3}\\ & \geq \|x-y\|-\epsilon
		\end{align*}
		and either
		\[|f_{a,b,c,m_0}(x) - f_{a,b,c,m_0}(y)| > \frac{1}{2}(\|x-y\| - \epsilon)\]
		or
		\[|f_{-a,b,c,m_0} (x) - f_{-a,b,c,m_0} (y)| > \frac{1}{2}(\|x-y\| - \epsilon).\]
		
		We deduce that the function $F:X\to \ell_\infty(\varLambda)$ satisfies
		\[\frac{1}{2}\|x-y\|\leq \|F(x)-F(y)\| \leq \|x-y\|\]
		for all $x,y\in X$, which establishes the result.
	\end{proof}
	
	Now, we are ready to prove above announced result.
	
	\begin{thm}\label{thm:c0Andellinfty}
		$\F((\bigoplus_{2^{\omega}}\ell_\infty)_{c_0})$ and $\F(\ell_\infty)$ are linearly isomorphic.
	\end{thm}
	\begin{proof}
		From Propositions \ref{Prop:C0ellinftyisanAbsoluteRetract} and \ref{Prop:EmbeddingC0(ellinfty)intoellinfty} we deduce that 
		$\F(\ell_\infty)$ has a complemented subspace isomorphic to $\F((\bigoplus_{2^{\omega}}\ell_\infty)_{c_0})$. Since it is evident that 
		$\F((\bigoplus_{2^{\omega}}\ell_\infty)_{c_0})$ has a complemented subspace isomorphic to $\F(\ell_\infty)$ and it is well known that 
		$\F(\ell_\infty)$ isomorphic to its $\ell_1$-sum (see \cite{Kau15}), the thesis follows by the Pe{\l}czy\'nski decomposition method.
	\end{proof}
	
	Let us mention the following characterization, which follows from the above mentioned results.
	
	\begin{cor}\label{cor:reduction}
		Let $M$ be a metric space and $X$ is one of the Banach spaces
		\[\ell_\infty,\quad  \Big(\bigoplus\nolimits_{2^{\omega}}\ell_\infty\Big)_{c_0},\quad \Big(\bigoplus\nolimits_{2^{\omega}}\ell_\infty\Big)_{\ell_1}.
		\]
		Then the following conditions are equivalent.
		\begin{enumerate}[label=(\alph*)]
			\item\label{it:sepLip} $(\Lip_0(M),w^*)$ is separable.
			\item\label{it:bddInj} There is a bounded linear injection of $\F(M)$ into $X$.
			\item\label{it:LipBddInj} There is a bounded linear injection of $\F(M)$ into $\F(X)$.
		\end{enumerate} 
	\end{cor}
	\begin{proof}Condition \ref{it:sepLip} implies there is a bounded linear injection of $\F(M)$ into $\ell_\infty$ which is isometric to a subspace of $X$, so \ref{it:bddInj} holds. Implication \ref{it:bddInj}$\Rightarrow$\ref{it:LipBddInj} follows from Theorem~\ref{thm:kaltonadaptation}. Finally, \ref{it:LipBddInj}$\Rightarrow$\ref{it:sepLip} follows easily from the fact that $\F(X)$ has a $w^*$-separable dual (for $X=\ell_\infty$ it is a direct consequence of Proposition~\ref{prop:dualUnitBall}, for $X = (\bigoplus\nolimits_{2^{\omega}}\ell_\infty)_{\ell_1}$ it follows from Propositions~\ref{prop:dualUnitBall} and \ref{prop:ell1Injection}, and finally for $X = (\bigoplus\nolimits_{2^{\omega}}\ell_\infty)_{c_0}$ it follows from the already proven case when $X=\ell_\infty$ and from Theorem~\ref{thm:c0Andellinfty}).
	\end{proof}
	
	\begin{remark}
		Note that similarly as Theorem~\ref{thm:c0Andellinfty} we may also obtain the following generalization. If $\lambda>0$ is given and $X_\alpha$, $\alpha<2^\omega$ are absolute $\lambda$-Lipschitz retracts which are $\lambda$-Lipschitz isomorphic to a subset of $\ell_\infty$, then $\F((\bigoplus_{\alpha<2^{\omega}} X_\alpha)_{c_0})$ is isomorphic to a complemented subspace of $\F(\ell_\infty)$.
		
		We note that it is open whether $\ell_\infty$ Lipschitz embeds into $c_0(2^\omega)$ (see \cite{Pelant} and also \cite{HJS24}, where more information may be found). If it was the case, then in the above we would even obtain that $\F((\bigoplus_{\alpha<2^{\omega}} X_\alpha)_{c_0})$ and $\F(\ell_\infty)$ are linearly isomorphic. We would also obtain that $\F(\ell_\infty)$ and $\F(c_0(2^\omega))$ are linearly isomorphic, which would solve \cite[Problem 237]{GMZ16} (see also \cite[Remark 13]{CK21}).
	\end{remark}
	
	\section{Concluding remarks}
	\label{sec-Conclusion}
	
	Let us mention several possible ways/reductions which might help proving Conjecture~\ref{conj:main}.
	\begin{itemize}
		\item Since it is known that $\ell_\infty/c_0$ contains isometric copy of all the Banach spaces of density $\omega_1$ and therefore under CH it is an isometrically universal Banach space of density continuum (see e.g. \cite{P63} and \cite{BK12} for more details), it seems to be interesting to find out whether $\F(\ell_\infty/c_0)$ has $w^*$-separable dual (if it does, then consistently every Lipschitz-free space of density at most continuum does as well).
		\item In Proposition~\ref{prop:sufficient} we formulated a general metric condition guaranteeing $w^*$-separability of $\Lip_0(M)$, but we were able to satisfy this condition only in locally separable complete metric spaces and we do not know whether this condition can be satisfied in every metric space of density continuum.
		\item For any metric space $M$ there exists a bounded metric space $B(M)$ such that both $M$ and $B(M)$ are topologically homeomorphic and $\F(M)$ is linearly isomorphic with $\F(B(M))$, see \cite{AACD22}. Thus, it would suffice to prove Conjecture~\ref{conj:main} for bounded metric spaces.
		\item Another possible approach could be to find for a metric space $M$ of density continuum some metric spaces $(M_i)_{i\in I}$ with $|I|\leq 2^\omega$ in such a way that each $\F(M_i)$ has $w^*$-separable dual and $\F(M)$ admits an injective operator into $\Big(\bigoplus_{I} \F(M_i)\Big)_{\ell_1}$ (if this is possible, then Corollary~\ref{cor:charactWeakStarSep} would imply that $\Lip_0(M)$ is $w^*$-separable).
	\end{itemize}
	
	While being interested in the solution of Conjecture~\ref{conj:main} as a byproduct of our considerations we discovered some facts which we find interesting.
	
	\begin{remark}
		Given a set $I$ with $|I|\leq 2^\omega$, $\F(C([0,\omega_1]))$ does not even uniformly embed into $\F(c_0(I))$. This can be deduced from the following two observations:
		\begin{itemize}
			\item first, we recall that $c_0(I)$ Lipschitz embeds into $\ell_\infty$ and therefore $\F(c_0(I))$ is isomorphic to a subspace of $\ell_\infty$ (see e.g. \cite[Proposition 2.9]{AGP24} or one can also use the fact that $\F(c_0(I))$ is linearly isomorphic to a subspace of $\F(\ell_\infty)$ which isometrically embeds into $\ell_\infty$ by Proposition~\ref{prop:dualUnitBall} or by \cite[Proposition 5.1]{Kal11});
			\item on the other hand, $C([0,\omega_1])$ does not uniformly embed into $\ell_\infty$ (see \cite[Theorem 4.2]{Kal11}), which implies that $\F(C([0,\omega_1]))$ does not uniformly embed into $\ell_\infty$ (since if $f:\F(C([0,\omega_1]))\to \ell_\infty$ is a uniform embedding, then $f\circ \delta:C([0,\omega_1])\to \ell_\infty$).
		\end{itemize}
		We find it interesting as in \cite[Corollary 4]{CG23} it is proved that $\F(C(K))$ and $\F(c_0(I))$ have linearly isomorphic duals whenever $K$ is a compact topological space with weight $|I|$.
	\end{remark}
	
	\begin{remark}\label{rem:sepNotSep}Our results give us as a byproduct quite a large list of Lipschitz free spaces which have a $w^*$-separable dual, but not $w^*$-separable dual unit ball. Indeed, it suffices to pick any metric space of density at most continuum which does not isometrically embed into $\ell_\infty$ (then $\F(M)$ does not have $w^*$-separable dual unit ball by Proposition~\ref{prop:dualUnitBall}) and for which we are able to prove that $(\Lip_0(M),w^*)$ is separable. Examples of such spaces include:
		\begin{itemize}
			\item the space $C([0,\omega_1])$ (which does not Lipschitz embed into $\ell_\infty$ by \cite[Theorem 4.2]{Kal11} and which has a projectional skeleton),
			\item any Banach space with a projectional skeleton of density continuum which contains a subspace isomorphic to $C([0,\omega_1])$,
			\item the uniformly discrete metric space $M$ constructed in \cite[Example 2.12]{AGP24}.
		\end{itemize}
	\end{remark}

\end{document}